\documentclass[a4paper]{amsart}

\usepackage{amsfonts}
\usepackage{amssymb}
\usepackage{amscd}
\usepackage{palatino}
\usepackage{amsthm}
\usepackage{amsmath}
\usepackage{a4wide}

\theoremstyle{plain}

\newtheorem{teo}{Theorem}[section]
\newtheorem{lemma}[teo]{Lemma}
\newtheorem{prop}[teo]{Proposition}

\newtheorem{ackn}{Acknowledgments\!}

\theoremstyle{definition}
\newtheorem{dfnz}[teo]{Definition}

\theoremstyle{remark}

\newtheorem{rem}[teo]{Remark}

\numberwithin{equation}{section}

\def\NN{{{\mathbb N}}}

\def\RR{{\mathbb R}}

\newcommand{\R}[1]{\mathbb{R}^{#1}}
\newcommand{\de}{\partial}
\newcommand{\ve}{\varepsilon}

\newcommand{\M}[1]{\mathcal{#1}}

\title[A Note on Quasilinear Parabolic Equations on Manifolds]{A Note on Quasilinear Parabolic Equations on Manifolds}
\date{\today}

\author[Carlo Mantegazza]{Carlo Mantegazza}
\address[Carlo Mantegazza]{Scuola Normale Superiore di Pisa, P.za Cavalieri 7, Pisa, Italy, 56126}
\email[C. Mantegazza]{c.mantegazza@sns.it}
\author[Luca Martinazzi]{Luca Martinazzi}
\address[Luca Martinazzi]{Centro di Ricerca Matematica ``Ennio De
  Giorgi'', Scuola Normale Superiore di Pisa, P.za Cavalieri 7, Pisa,
  Italy, 56126}
\email[L. Martinazzi]{luca.martinazzi@sns.it}

\date{\today}

\begin{document}

\begin{abstract}  We prove short time existence,
    uniqueness and continuous dependence on the initial data of smooth
    solutions of quasilinear locally parabolic equations of arbitrary
    even order on closed manifolds.
\end{abstract}

\maketitle

\tableofcontents

\section{Introduction}

Let $(M,g)$ be a compact, smooth Riemannian manifold without boundary
of dimension $n$ and let $d\mu$ be the canonical measure associated to the 
metric tensor $g$.

We consider the parabolic problem with a smooth initial datum $u_0:M\to\RR$,
\begin{equation}\label{eq0}
\left\{
\begin{array}{ll}
u_t=Q[u]& \text{in }M\times [0,T]\\
u(\,\cdot\,,0)=u_0&\text{on } M\,,
\end{array}
\right.
\end{equation}
where $Q$ is a smooth, quasilinear, locally elliptic operator of order
$2p$, defined in $M\times[0,{\mathcal T})$ for some ${\mathcal T}>0$ which, 
adopting (as in the rest of the paper) the Einstein convention
of summing over repeated indices, can be expressed in local 
coordinates as
$$Q[u](x,t)=A^{i_1\dots i_{2p}}(x,t,u,\nabla u,\dots
,\nabla^{2p-1} u)\nabla^{2p}_{i_1\dots i_{2p}}u(x,t)+ b(x,t,u,\nabla u,\dots ,\nabla^{2p-1} u)\,,$$
where $A$ is a locally elliptic smooth $(2p,0)$--tensor of the form
\begin{equation}\label{struct}
A^{i_1 j_1 \dots i_pj_p}= (-1)^{p-1}E_1^{i_1 j_1}\cdots E_p^{i_pj_p}
\end{equation}
for some $(2,0)$--tensors $E_1,\dots,E_p$ and a function $b$ smoothly
depending on their arguments.\\
Local ellipticity here means that for every $L>0$ there exists a
positive constant $\lambda\in\RR$ such that
each tensor $E_\ell$ satisfies
\begin{equation}\label{ellit}
E_\ell^{ij}(x,t,u,\psi_1,\ldots,\psi_{2p-1})\xi_i\xi_j\ge \lambda |\xi|_{g(x)}^2,\quad \text{for every }
\xi\in T^*_xM\,,
\end{equation}
when $x\in M$, $t\in [0,T]$ with $T<{\mathcal T}$, $u\in \R{}$ with $|u|\le L$,
$\psi_k\in\otimes^kT^*_xM$ with $|\psi_k|_{g(x)}\le L$. In other words we require
that condition~\eqref{ellit} holds for some positive $\lambda$ whenever the
arguments of $E^{ij}_\ell$ lie in a compact set $K$ of their natural
domain of definition and assume that $\lambda$ depends only on
$K$. If $\lambda>0$ can be chosen independent of $K$ (i.e. of $L$), 
then we shall say that $A$ is \emph{uniformly} elliptic.

Clearly, this is not the most general notion of quasilinear parabolic
problems, due to the special ``product'' structure of the operator,
anyway it covers several important situations. 
For instance, our definition includes the case of standard
locally parabolic equations of order two in
non--divergence form. Notice that we make no growth assumptions on the
tensor $A$ and the function $b$.

Interchanging covariant derivatives, integrating by parts and using
interpolation inequalities (see~\cite{polden2} for details), 
the following G\r{a}rding's inequality holds for this class of operators. For every
smooth $u$ and $t\in[0,{\mathcal T})$, we have 
\begin{equation}\label{garding} 
-\int_M \psi A^{i_1\dots i_{2p}}(u)\nabla^{2p}_{i_1\dots i_{2p}}\psi\,d\mu
\geq \sigma \Vert \psi \Vert^2_{W^{p,2}(M)}- C \Vert
\psi \Vert^2_{L^2(M)}\qquad\forall\psi\in C^\infty(M)\,,
\end{equation}
where the constants $\sigma>0$ and $C>0$ depend continuously 
only on the $C^p$--norm of the tensor $A$ and on the $C^{3p-1}$--norm of the function $u$ at time
$t$ (and on the curvature tensor of $(M,g)$ and its covariant
derivatives). In particular, if $u$ depends smoothly on time,
$\sigma=\sigma(t)$ and $C=C(t)$ are continuous functions of time.

The aim of this note is to prove the following short time existence result. 
\begin{teo}\label{main} For every $u_0\in C^\infty(M)$ there exists a 
positive time $T>0$ such that problem~\eqref{eq0} has a smooth solution. 
Moreover, the solution is unique and depends continuously on $u_0$ in
the $C^\infty$--topology.
\end{teo}

\medskip

Our interest in having a handy proof of this result is related to
geometric evolution problems, like for instance the Ricci flow, the mean curvature
flow, the Willmore flow~\cite{kuschat1}, the $Q$--curvature
flow~\cite{malstr}, the Yamabe flow~\cite{brendle1,struwschwe,ye1}, etc. 
In all these problems, the very first step is
to have a short time existence theorem showing that for an initial
geometric structure (hypersurface, metric) the flow actually starts. Usually, after
some manipulations in order to eliminate the degeneracies due to the
geometric invariances, one has to face a quasilinear parabolic equation
with smooth coefficients and smooth initial data.

If we replace the compact manifold $M$ with a
bounded domain $\Omega\subset\R{n}$, short time existence for
quasilinear systems of order two, with prescribed boundary conditions
and initial data, was proven by Giaquinta and Modica~\cite{giaqmod} in
the setting of H\"older spaces.

A different approach to Theorem~\ref{main} was developed by Polden in his 
Ph.D. Thesis~\cite{polden2} (see also~\cite{huiskpold}), by means of 
an existence result for linear equations in parabolic Sobolev
spaces and the inverse function theorem. Unfortunately, as pointed out by
Sharples~\cite{sharples}, such procedure has a gap in the convergence
of the solutions of the ``frozen'' linear problems to a solution of
the quasilinear one.

In the same paper~\cite{sharples} Sharples, pushing
further the estimates of Polden and allowing nonsmooth coefficients, 
was able by means of an iteration scheme to show the existence 
of a short time solution of the quasilinear 
problem on a two--dimensional manifold, when the operator is of
order two and {\em in divergence form}.

Our goal here is instead to simply fill the gap in Polden's proof. We start
with his linear result and we show that his linearization
procedure actually works if one linearizes at a suitably chosen function and 
discusses in details the above mentioned convergence.\\
As we do not assume any condition on the
operator (only its product structure) and on the dimension of the
manifold, we have a complete proof of the short time existence 
of a smooth solution to these quasilinear locally parabolic equations of arbitrary order
on compact manifolds and of its uniqueness and smooth dependence 
on the initial data.
We refer the interested reader to the nice and detailed introduction
in~\cite{sharples} for the different approaches to the problem.

\medskip

The paper is organized as follows.
In the next section we present the linearization procedure, assuming
Polden's linear result (Proposition~\ref{teopolden} below) and we
prove Theorem~\ref{main} by means of 
Lemma~\ref{F} which is the core of our argument. Roughly speaking,
when a candidate solution $u$ stays in some parabolic Sobolev space
of order high enough, the functions $u, \nabla u,\dots,\nabla^{2p-1}u$
are continuous (or even more regular), hence the same holds for the
tensor $A$ and the function $b$. This implies that the map $u\mapsto
(u_t-Q[u])$ is of class $C^1$ between some suitable spaces, as it
closely resembles a linear map with regular coefficients. This allows 
the application of the inverse function theorem which, in conjunction
with an approximation argument, yields the existence of a solution.
The last two sections are devoted to the proof of Lemma~\ref{F} and
to the discussion of the parabolic Sobolev embeddings on which such
proof relies.

We mention that the results can be extended to quasilinear
parabolic {\em systems} as the linearization procedure remains the
same and Polden's linear estimates (Proposition~\ref{teopolden}) can
be actually easily generalized, assuming a suitable 
definition of ellipticity. In fact one easily sees
that our result applies to all quasilinear systems whose linearization
is invertible in the sense of Proposition~\ref{teopolden2} below. For
more general definition of elliptic or parabolic operators of
higher--order see~\cite{AgDoNi}.

\bigskip

{\em In the following the letter $C$ will denote a constant which can
  change from a line to another and even within the same formula.}

\bigskip

\begin{ackn} 
We are grateful to Alessandra~Lunardi for useful suggestions.\\
We wish to thank Mariano~Giaquinta for several interesting discussions.\\ 
The authors are partially supported by the Italian project FIRB--IDEAS
``Analysis and Beyond''.\\
The second author is partially supported by the 
{\em Swiss National Fond} Grant n.~PBEZP2--129520.
\end{ackn}

\section{Proof of the Main Theorem}

We recall Polden's result for linear parabolic equations.
Let us consider the problem
\begin{equation} \label{sistemalin}
\begin{cases}
\displaystyle{u_t - A^{i_1\dots i_{2p}}\nabla^{2p}_{i_1\dots
    i_{2p}}u
-\sum_{k=0}^{2p-1}R_k^{j_1\dots j_k}\nabla^k_{j_1\dots j_k}u=b}\\
\displaystyle{u(\,\cdot\,,0) = u_0 \,,}
\end{cases}
\end{equation}
where all the tensors $A$ and $R_k$ depend only on $(x,t)\in M\times [0,+\infty)$, are 
smooth and uniformly bounded with all their
derivatives. Moreover, we assume that
the tensor $A$ has the product structure~\eqref{struct}, where each
$E_\ell$ is uniformly elliptic.

The G\r{a}rding's inequality for the linear operator
$$
L(u)=A^{i_1\dots i_{2p}}\nabla^{2p}_{i_1\dots
    i_{2p}}u
-\sum_{k=0}^{2p-1}R_k^{j_1\dots j_k}\nabla^k_{j_1\dots j_k}u
$$ 
reads (see again~\cite{polden2} for details) 
\begin{equation}\label{garding2} 
-\int_M \psi L(\psi)\,d\mu
\geq \frac{\lambda}{2} \Vert \psi \Vert^2_{W^{p,2}(M)}- C \Vert
\psi \Vert^2_{L^2(M)}\qquad\forall\psi\in C^\infty(M)\,,
\end{equation}
where the constant $C>0$ depends only on the $C^p$--norm of the
tensors $A$ and $R_k$. Clearly, by approximation this inequality holds
also for every $\psi\in W^{2p,2}(M)$.

\begin{dfnz}
For any $m\in\mathbb{N}$ and $a\in\RR^+$ we define $P^m_a(M)$ to be the
completion of $C_c^\infty(M\times[0,+\infty))$ under the {\em parabolic norm}
$$
\Vert f\Vert_{P^m_a(M)}^2=\sum_{\text{$j,k\in\NN$ and $2pj+k\leq2pm$}}\int_{M\times[0,+\infty)}
e^{-2at}\vert\de_t^j \nabla^k f\vert^2\,d\mu\,dt
$$
and analogously $P^m(M,T)$ as the completion of
$C^\infty(M\times[0,T])$ under the norm
$$
\Vert f\Vert_{P^m(M,T)}^2=\sum_{\text{$j,k\in\NN$ and
    $2pj+k\leq2pm$}}\int_{M\times[0,T]}
\vert\de_t^j \nabla^k f\vert^2\,d\mu\,dt\,,
$$
for every $T\in\RR^+$.
\end{dfnz}
Clearly  for every $T\in\RR^+$ there is a natural continuous embedding
$P^m_a(M)\hookrightarrow P^m(M,T)$.

We have then the following global existence result for problem~\eqref{sistemalin}, by
Polden~\cite[Theorem~2.3.5]{polden2}.
\begin{prop}\label{teopolden} For every $m\in\NN$ there exists
  $a\in\RR^+$ 
large enough such that the linear map
\begin{equation}\label{Flin}
\Phi(u)=\Bigl(u_0, u_t - A^{i_1\dots i_{2p}}\nabla^{2p}_{i_1\dots
    i_{2p}}u
-\sum_{k=0}^{2p-1}R_k^{j_1\dots j_k}\nabla^k_{j_1\dots j_k}u\Bigr)=(u_0,L(u))\,,
\end{equation}
where $u_0=u(\,\cdot\,,0)$, is an isomorphism of $P_a^m(M)$ onto
$W^{p(2m-1),2}(M)\times P_a^{m-1}(M)$.
\end{prop}

In the following it will be easier (though conceptually equivalent) to
use the spaces $P^m(M,T)$ instead of the weighted spaces $P_a^m(M)$. For
this reason we translate Proposition~\ref{teopolden} into the setting 
of $P^m(M,T)$ spaces.

\begin{prop}\label{teopolden2} For every $T>0$ and
  $m\in\mathbb{N}$ the map $\Phi$ given by formula~\eqref{Flin}
  is an isomorphism of $P^m(M,T)$ onto $W^{p(2m-1),2}(M)\times P^{m-1}(M,T)$.
\end{prop}

\begin{proof}
The continuity of the second component of $\Phi$ is obvious while the
continuity of the first component follows as in the Polden's proof of
Proposition~\ref{teopolden} in~\cite{polden2}. Hence, the map $\Phi$
is continuous, now we show that it is an isomorphism.\\ 
Given any $b\in P^{m-1}(M,T)$ we consider an 
extension $\widetilde{b}\in P^{m-1}_a(M)$ of the 
function $b$ and we let $\widetilde{u}\in P^m_a(M)$ be the 
solution of problem~\eqref{sistemalin} for $\widetilde{b}$. Clearly,
$u=\widetilde{u}\vert_{M\times[0,T]}$ belongs to $P^m(M,T)$ and
satisfies $\Phi(u)=(u_0,b)$ in $M\times[0,T]$. Suppose
that $v\in P^m(M,T)$ is another 
function such that $\Phi(v)=(u_0,b)$ in $M\times[0,T]$, then setting $w=u-v\in P^m(M,T)$ we have 
that 
\begin{equation*}
\begin{cases}
\displaystyle{w_t - A^{i_1\dots i_{2p}}\nabla^{2p}_{i_1\dots
    i_{2p}}w
-\sum_{k=0}^{2p-1}R_k^{j_1\dots j_k}\nabla^k_{j_1\dots j_k}w=w_t-L(w)=0}\\
\displaystyle{w(\,\cdot\,,0) = 0 \,.}
\end{cases}
\end{equation*}
By the very definition of solution in $P^m(M,T)$ (see~\cite{polden2}) 
and G\r{a}rding's inequality~\eqref{garding2}, we get
\begin{align*}
\int_M w^2(x,t)\,d\mu(x)=&\,\int_0^t\int_M 2ww_t\,d\mu\,ds\\
=&\,2\int_0^t\int_M wL(w)\,d\mu\,ds\\
\leq&\,-\frac{\lambda}{2} \int_0^t\int_M \vert \nabla^p
w\vert^2\,d\mu\,ds+C\int_0^t\int_M w^2\,d\mu\,ds\\
\leq&\, C\int_0^t\int_M
w^2(x,s)\,d\mu(x)\,ds\,,
\end{align*}
as $w(\,\cdot\,,t)\in W^{2p,2}(M)$ for almost every $t\in[0,T]$ and 
where the constant $C>0$ depends only on $T$ as the coefficients of
the operator $L$ are smooth. Then, by Gronwall's lemma (in its integral
version) it follows that $\int_M w^2(\,\cdot\,,t)\,d\mu$ is zero for
every $t\in[0,T]$, as it is zero at time $t=0$. It follows
that $w$ is zero on all $M\times[0,T]$, hence the two functions $u$ and
$v$ must coincide.

Since the map $\Phi:P^m(M,T)\to W^{p(2m-1),2}(M)\times P^{m-1}(M,T)$ is
continuous, one--to--one and onto, it is an isomorphism by the open mapping theorem.
\end{proof}

\begin{rem}\label{smrem} 
When $u_0$ and $b$ are smooth 
the unique solution $u$ of problem~\eqref{sistemalin} belongs to all
the spaces $P^m(M,T)$ for every $m\in\NN$. As by Sobolev
embeddings for any $k\in \mathbb{N}$ we can find a large 
$m\in\mathbb{N}$ so that $P^m(M,T)$ continuously embeds into
$C^k(M\times[0,T])$,
we can conclude that $u$ actually belongs to $C^\infty(M\times [0,T])$.
\end{rem}

Now we are ready to prove Theorem~\ref{main}. The tensor
$A$ and the function $b$ from now on will depend on
$x,t,u,\nabla u,\ldots,\nabla^{2p-1}u$ as in the introduction. 
Since $M$ is compact there exists a constant $C>0$ such that the initial
datum satisfies $\vert u_0\vert+\vert\nabla u_0\vert_g+\ldots+|\nabla^{2p-1} u_0|_g\leq C$. 
Then, since we are interested in existence for short time, possibly modifying 
the tensor $A$ and the function $b$ outside a compact set with some ``cut--off'' functions, 
we can assume that if $\vert u\vert+\vert\nabla u\vert_g +\ldots +|\nabla^{2p-1}u|_g +t \geq2C$, then
$$
E_\ell^{ij}(x,t,u,\nabla u,\dots,\nabla^{2p-1} u)=g^{ij}(x),\quad
\text{ and }\,\quad b(x,t,u,\nabla u,\dots,\nabla^{2p-1} u)=0\,.
$$ 
In particular we can assume that the tensors $E_\ell$ are uniformly
elliptic.

For a fixed $m\in\NN$, we consider the map defined on $P^m(M,T)$ given by
\begin{equation*}
\M{F}(u)=(u_0, u_t-Q[u])
=\Bigl(u(\,\cdot\,,0), u_t - A(u)\cdot\nabla^{2p}u 
-b(u)\Bigr)\,,
\end{equation*}
where in order to simplify we used the notation
$$
A(u)\cdot \nabla^{2p}v(x,t)=A^{i_1\dots i_{2p}}(x,t,u(x,t),\dots,
\nabla^{2p-1}u(x,t))\nabla^{2p}_{i_1\dots i_{2p}}v(x,t)\,,
$$
and
$$
b(u)(x,t)=b(x,t,u(x,t),\dots,\nabla^{2p-1}u(x,t))
$$
for $u,v\in P^{m}(M,T)$.\\
We have seen in Proposition~\ref{teopolden2} that if $A(u)$ and $b(u)$
only depend on $x\in M$ and $t\in[0,T]$ (and not on $u$ and its space
derivatives), then $\M{F}$ is a continuous map from 
$P^m(M)$ onto $W^{p(2m-1),2}(M)\times P^{m-1}(M)$. This is not the
case in general when $A$ and $b$ depend on $u$ and its derivatives,
but it is true if $m\in\NN$ is large enough and in this case $\M{F}$
is actually $C^1$.

\begin{lemma}\label{F} Assume that
\begin{equation}\label{2pm}
m>\frac{\mathrm{dim}M+6p-2}{4p}=\frac{n+6p-2}{4p}\,,
\end{equation}
and $u\in P^{m}(M,T)$. Then $\M{F}(u)\in W^{p(2m-1),2}(M)\times P^{m-1}(M,T)$ and the map 
$$
\M{F}:P^m(M,T)\to W^{p(2m-1),2}(M)\times P^{m-1}(M,T)
$$
is of class $C^1$.
\end{lemma}
We postpone the proof of this lemma to Section~\ref{lemmasec}.

We fix $m\in\NN$ such that the hypothesis of Lemma~\ref{F} holds and we set 
$$
\widetilde{u}_0(x,t)=\sum_{\ell=0}^{m-1}\frac{a_\ell(x)\,t^\ell}{\ell!}
$$
for some functions $a_0,\dots,a_{m-1}\in C^\infty(M)$ to be determined
later. Let $w\in P^m(M,T)$ be the unique solution of the linear
problem
$$
\begin{cases}
\displaystyle{w_t=A(\widetilde{u}_0)\cdot \nabla^{2p}w +
b(\widetilde{u}_0)}\\
\displaystyle{w(\,\cdot\,,0)=u_0\,.}
\end{cases}
$$
Such solution exists by Proposition~\ref{teopolden2} and it is
smooth by Remark~\ref{smrem}, as $u_0$ and $\widetilde{u}_0$ are smooth 
(thus also $A(\widetilde{u}_0)$ and $b(\widetilde{u}_0)$).\\
Hence, we have
$$
\M{F}(w)=(u_0,w_t-Q[w])
=\Bigl(u_0,(A(\widetilde{u}_0)-A(w))\cdot \nabla^{2p}w
+b(\widetilde{u}_0)-b(w)\Bigr)=:(u_0,f)\,,
$$
where we set
$f=(A(\widetilde{u}_0)-A(w))\cdot \nabla^{2p}w+
b(\widetilde{u}_0)-b(w)$.

If we compute the differential $d\M{F}_w$ of the map $\M{F}$
at the ``point'' $w\in C^\infty(M\times[0,T])$, acting on $v\in P^m(M,T)$, we
obtain
\begin{align}\label{dfeq}
d\M{F}_w(v)=\Bigl(v_0, v_t&\,-A^{i_1\dots i_{2p}}(w)\nabla^{2p}_{i_1\dots
  i_{2p}}v-D_w A^{i_1\dots i_{2p}}(w)v\nabla^{2p}_{i_1\dots
  i_{2p}}w\dots\\
&\,\dots-D_{w_{j_1\dots
  j_{2p-1}}} A^{i_1\dots i_{2p}}(w){\nabla^{2p-1}_{j_1\dots
  j_{2p-1}}}v\nabla^{2p}_{i_1\dots
  i_{2p}}w\nonumber\\
&\,-D_w b(w)v\dots-D_{w_{j_1\dots
  j_{2p-1}}} b(w){\nabla^{2p-1}_{j_1\dots
  j_{2p-1}}}v\Bigr)\,,\nonumber
\end{align}
where $v_0=v(\,\cdot\,,0)$ and we denoted by $D_{w_{j_1\dots
  j_k}} A^{i_1\dots i_{2p}}(w)$, $D_{w_{j_1\dots
  j_k}} b(w)$ the derivatives of the tensor $A$ and of the function $b$
with respect to their variables $\nabla^{k}_{j_1\dots j_k}w$, respectively.\\
Then, we can see that $d\M{F}_w(v)=(z,h)\in W^{p(2m-1),2}(M)\times
P^{m-1}(M,T)$ implies that $v$ is a solution of
the linear problem 
$$
\begin{cases}
\displaystyle{v_t-\widetilde{A}^{i_1\dots i_{2p}}\nabla^{2p}_{i_1\dots
  i_{2p}}v-\sum_{k=0}^{2p-1}\widetilde{R}_k^{j_1\dots
  j_k}\nabla^k_{j_1\dots j_k}v=h}\\
\displaystyle{v(\,\cdot\,,0) = z \,,}
\end{cases}
$$
where 
\begin{align*}
\widetilde{A}^{i_1\dots i_{2p}}=&\,A^{i_1\dots
    i_{2p}}(w)\,,\\
\widetilde{R}_k^{j_1\dots
  j_k}=&\,D_{w_{j_1\dots
  j_k}} A^{i_1\dots i_{2p}}(w)\nabla^{2p}_{i_1\dots
  i_{2p}}w+D_{w_{j_1\dots
  j_k}} b(w)
\end{align*}
are smooth tensors independent of $v$.\\
By Proposition~\ref{teopolden2} for every $(z,h)\in W^{p(2m-1),2}(M)\times
P^{m-1}(M,T)$ there exists a unique solution $v$ of this problem, 
hence $d\M{F}_w$ is a Hilbert space isomorphism and the inverse
function theorem can be applied, as the map $\M{F}$ is $C^1$ by Lemma~\ref{F}.
Hence, the map $\M{F}$ is a diffeomorphism 
of a neighborhood $U\subset P^m(M,T)$ of $w$ onto 
a neighborhood $V\subset W^{p(2m-1),2}(M)\times P^{m-1}(M,T)$ of
$(u_0,f)$.

Getting back to the functions $a_\ell$, we claim that we can choose them such that 
$a_\ell=\partial_t^\ell w\vert_{t=0}\in C^\infty(M)$ for every
$\ell=0,\dots,m-1$.\\ 
We apply the following recurrence procedure. We set $a_0=u_0\in
C^\infty(M)$ and, assuming to have defined $a_0,\dots, a_\ell$, we 
consider the derivative 
$$
\partial_t^{\ell+1}w\vert_{t=0}=\partial_t^\ell[A^{i_1\dots i_{2p}}(x,t,\widetilde{u}_0,\nabla\widetilde{u}_0,\dots
,\nabla^{2p-1}\widetilde{u}_0)\nabla^{2p}_{i_1\dots i_{2p}}w +
b(x,t,\widetilde{u}_0,\nabla\widetilde{u}_0,\dots,\nabla^{2p-1}\widetilde{u}_0)]\Bigr\vert_{t=0}
$$
and we see that the right--hand side contains time-derivatives at time
$t=0$ of $\widetilde{u}_0,\dots,\nabla^{2p-1}\widetilde{u}_0$ and
$\nabla^{2p}_{i_1\dots i_{2p}}w$ only up to the order
$\ell$, hence it only depends on the functions $a_0,\dots,a_\ell$. 
Then, we define $a_{\ell+1}$ to be equal to such
expression. Iterating up to $m-1$, the set of functions $a_0,\dots,a_{m-1}$ satisfies
the claim.

Then, $a_\ell=\partial_t^\ell\widetilde{u}_0\vert_{t=0}=\partial_t^\ell
w\vert_{t=0}$ and it easily follows by the ``structure'' of the
function $f\in C^\infty(M\times[0,T])$, that we 
have $\partial_t^\ell f\vert_{t=0}=0$ and 
$\nabla^j\partial_t^\ell f\vert_{t=0}=0$ for any
$0\leq \ell\leq m-1$ and $j\in\NN$.

We consider now for any $k\in\NN$ 
the ``translated'' functions $f_k:M\times[0,T]\to\RR$ given by 
$$
f_k(x,t)=
\begin{cases}
0 &\text{\, if $t<1/k$}\\
f(x,t-1/k) \,\,&\text{\, if\, $1/k\leq t\leq T$}\,.
\end{cases}
$$
Since $f\in C^\infty(M\times[0,T])$ and
$\nabla^j\partial_t^\ell f\vert_{t=0}=0$ for every $0\leq\ell\leq m-1$ and every
$j \in\NN$, all the functions
$\nabla^j\partial_t^\ell f_k\in C^0(M\times [0,T])$ for every $0\le \ell \le m-1$ and $j\ge 0$, it follows easily that
$$
\nabla^j\partial_t^\ell f_k\to \nabla^j\partial_t^\ell f\quad \text{in
  $L^2(M\times [0,T])$ for $0\le \ell\le m-1$, $j\ge 0$}\,,
$$
hence $f_k\to f$ in $P^m(M,T)$.

Hence, there exists a function 
$\widetilde{f}\in P^{m-1}(M,T)$ such that $(u_0,\widetilde{f})$ belongs
to the neighborhood $V$ of $\M{F}(w)$ and $\widetilde{f}=0$ in
$M\times[0,T^\prime]$ for some $T^\prime\in (0,T]$. 
Since $\M{F}\vert_U$ is a diffeomorphism between $U$ and $V$, 
we can find a function $u\in U$ such that
$\M{F}(u)=(u_0, \widetilde{f})$. Clearly such $u\in P^m(M,T^\prime)$ is a solution
of problem~\eqref{eq0} in $M\times[0,T^\prime]$. 
Since $u\in P^m(M,T^\prime)$ implies that $\nabla^{2p-1}u\in C^0(M\times[0,T^\prime])$, parabolic
regularity implies that actually $u\in C^\infty(M\times [0,T^\prime])$.

\medskip

We now prove uniqueness by a standard energy estimate, which we
include for completeness. In the sequel for simplicity 
we relabel $T$ the time $T^\prime$ found above.

Suppose that we have two smooth solutions
$u,v:M\times[0,T]\to\RR$ of Problem~\eqref{eq0}. 
Setting $w:=u-v$, we compute in an orthonormal frame
\begin{align*}
\frac{d\,}{dt}\int_M\vert \nabla^pw \vert^2\,d\mu=&\,\int_M 2\nabla^p_{i_1\dots i_p}w\nabla^p_{i_1\dots
  i_p}\Bigl(A(u)\cdot \nabla^{2p}u-A(v)\cdot \nabla^{2p}v\Bigr)\,d\mu\\
&\,+\int_M 2\nabla^p_{i_1\dots i_p}w\nabla^p_{i_1\dots i_p}\Bigl(b(u)-b(v)\Bigr)\,d\mu\\
=&\,2\int_M \nabla^p_{i_1\dots i_p}w\nabla^p_{i_1\dots
  i_p}\Bigl(A(u)\cdot \nabla^{2p} w\Bigl)\,d\mu\\
&\,+2\int_M \nabla^p_{i_1\dots i_p}w\nabla^p_{i_1\dots
  i_p}\Bigl((A(u)-A(v))\cdot \nabla^{2p} v\Bigr)\,d\mu\\
&\,+2(-1)^p\int_M\nabla^{p}_{i_1\cdots i_p}\nabla^{p}_{i_1\cdots i_p}w\Bigl(b(u)-b(v)\Bigr)\,d\mu\\
\leq&\,2\int_M \nabla^p_{i_1\dots i_p}w\nabla^p_{i_1\dots
  i_p}\Bigl(A(u)\cdot\nabla^{2p}w\Bigr)\,d\mu\\
&\,+2\int_M\vert\nabla^{2p}w\vert\Bigl
(\vert A(u)-A(v)\vert\,\vert\nabla^{2p}v\vert+\vert b(u)-b(v)\vert\Bigr)\,d\mu\,,
\end{align*}
where the integrals over $M$ are intended at time $t\in [0,T]$.\\
Now we consider the integral $\int_M \nabla^p_{i_1\dots i_p}w
\nabla^p_{i_1\dots
  i_p}(A^{j_1\dots j_{2p}}(u)\nabla^{2p}_{j_1\dots
  j_{2p}}w)\,d\mu$. Expanding the derivative 
$\nabla^p_{i_1\dots
  i_p}(A^{j_1\dots j_{2p}}(u)\nabla^{2p}_{j_1\dots j_{2p}}w)
$ we will get one special term $A^{j_1\dots j_{2p}}(u)\nabla^{3p}_{i_1\dots
  i_pj_1\dots j_{2p}}w$ and several other terms of the form
$B(x,t,u,\dots,\nabla^{3p-1}u)\#\nabla^qw$ with $2p\leq q<3p$, for
some tensor $B$ smoothly depending on its arguments, where the symbol
$\#$ means metric 
contraction on some indices. For each of these terms, integrating
repeatedly by parts, we can write
$$
\int_M \nabla^pw \# B(x,t,u,\dots,\nabla^{3p-1}u)\#\nabla^qw\,d\mu=\sum_{\ell=p}^{2p}\int_M \nabla^{\ell}w\#
D_\ell(x,t,u,\dots,\nabla^{4p-1}u)\#\nabla^{q-p}w\,d\mu
$$
where the tensors $D_\ell$ are smoothly
depending on their arguments.\\
Since $u\in C^\infty(M\times[0,T])$, all the tensors $D_\ell$ are bounded, hence we can estimate
\begin{align*}
\int_M \nabla^p_{i_1\dots i_p}w
\nabla^p_{i_1\dots
  i_p}(A(u)\cdot \nabla^{2p}w)\,d\mu\leq
&\,\int_M \nabla^p_{i_1\dots i_p}w A^{j_1\dots j_{2p}}(u)\nabla^{3p}_{i_1\dots
  i_pj_1\dots j_{2p}}w\,d\mu\\
&\,+C\sum_{r=p}^{2p-1}\sum_{\ell=p}^{2p}
\int_M \vert\nabla^\ell w\vert\,\vert\nabla^rw\vert\,d\mu\,.
\end{align*}
where $C$ is a constant independent of time 
(actually $C$ depends only on the structure of $A$).
Interchanging the covariant derivatives we have
$$
\nabla^{3p}_{i_1\dots
  i_pj_1\dots j_{2p}}w=
\nabla^{3p}_{j_1\dots
  i_{2p}i_1\dots i_p}w+ \sum_{q=0}^{3p-1}R_q\#\nabla^qw
$$
where the tensors $R_q$ are functions of the Riemann tensor and its
covariant derivatives, hence they are smooth and bounded. We can
clearly deal with this sum of terms as above, by means of
integrations by parts, obtaining the same result. 
Then we conclude, also using G\r{a}rding's inequality~\eqref{garding}
\begin{align*}
\int_M \nabla^p_{i_1\dots i_p}w
\nabla^p_{i_1\dots i_p}(A(u)\cdot\nabla^{2p}w)\,d\mu\leq
&\,\int_M \nabla^p_{i_1\dots i_p}w 
A^{j_1\dots j_{2p}}(u)\nabla^{2p}_{j_1\dots j_{2p}}\nabla^p_{i_1\dots i_p}w\,d\mu\\
&\,+C\sum_{r=p}^{2p-1}\sum_{\ell=p}^{2p}
\int_M \vert\nabla^\ell w\vert\,\vert\nabla^rw\vert\,d\mu\\
\le & -\alpha\int_M \vert\nabla^{2p}w\vert^2\,d\mu +C\sum_{r=p}^{2p-1}\sum_{\ell=p}^{2p}
\int_M \vert\nabla^\ell w\vert\,\vert\nabla^rw\vert\,d\mu\,,
\end{align*}
for some positive constant $\alpha$. Getting back to the initial
computation and using Peter--Paul inequality we get
\begin{align*}
\frac{d\,}{dt}\int_M\vert \nabla^pw\vert^2\,d\mu
\leq&\,-2\alpha\int_M \vert\nabla^{2p}w\vert^2\,d\mu 
+C\sum_{r=p}^{2p-1}\sum_{\ell=p}^{2p}\int_M \vert\nabla^{\ell}w\vert\,\vert\nabla^rw\vert\,d\mu\\
&\,+C\int_M\vert\nabla^{2p}w\vert\Bigl
(\vert
A(u)-A(v)\vert\,\vert\nabla^{2p}v\vert+\vert
b(u)-b(v)\vert\Bigr)\,d\mu\\
\leq&\,-2\alpha\int_M
\vert\nabla^{2p}w\vert^2\,d\mu+C\sum_{r=p}^{2p-1}\sum_{\ell=p}^{2p-1}
\int_M |\nabla^\ell w|\vert\nabla^rw\vert\,d\mu\\
&\,+ \sum_{r=0}^{2p-1}\Bigl(
\varepsilon_r \int_M \vert\nabla^{2p}w\vert^2\,d\mu 
+C_{\varepsilon_r}\int_M|\nabla^r w|^2d\mu\Bigr)\\
&\,+\delta \int_M \vert\nabla^{2p}w\vert^2\,d\mu + C_\delta\int_M\Bigl
(\vert A(u)-A(v)\vert^2+\vert
b(u)-b(v)\vert^2\Bigr)\,d\mu \\
\leq&\,-\alpha\int_M \vert\nabla^{2p}w\vert^2\,d\mu
+C \sum_{r=0}^{2p-1}\int_M \vert\nabla^rw\vert^2\,d\mu\\
&\,+C_\delta \int_M\Big(\vert A(u)-A(v)\vert^2+\vert
b(u)-b(v)\vert^2\Big)\,d\mu\,,
\end{align*}
where we chose $\delta+\sum_{r=0}^{2p-1}\varepsilon_r=\alpha$ and we
used the fact that $\vert\nabla^{2p}v\vert$ is bounded.\\
As the tensor $A$ and the function $b$ are smooth, we can easily bound
$$
\vert A(u)-A(v)\vert^2+\vert
b(u)-b(v)^2\vert\leq C\sum_{r=0}^{2p-1}\vert\nabla^r
u-\nabla^rv\vert^2=C\sum_{r=0}^{2p-1}\vert\nabla^r w\vert^2\,,
$$
so finally
$$
\frac{d\,}{dt}\int_M\vert \nabla^pw \vert^2\,d\mu
\leq-\alpha\int_M \vert\nabla^{2p}w\vert^2\,d\mu
+C\sum_{r=0}^{2p-1}\int_M\vert\nabla^rw\vert^2\,d\mu\,.
$$
Now we have, using again G\r{a}rding's and Peter--Paul inequalities,
\begin{align*}
\frac{d\,}{dt}\int_M w^2\,d\mu
=&\,2\int_M w\Bigl(A(u)\cdot\nabla^{2p}u-A(v)\cdot\nabla^{2p}v\Bigr)\,d\mu+2\int_M w\Bigl(b(u)-b(v))\Bigr)\,d\mu\\
=&\,2\int_M w A(u)\cdot \nabla^{2p}w\,d\mu+2\int_M
w\Big((A(u)-A(v))\cdot\nabla^{2p}v+b(u)-b(v)\Big)\,d\mu\\
\leq&\, -\beta\int_M \vert\nabla^pw\vert^2\,d\mu + C\int_M w^2\,d\mu+ C\int_M
w(A(u)-A(v)+b(u)-b(v))\,d\mu\\
\leq&\, -\beta\int_M \vert\nabla^pw\vert^2\,d\mu + C\int_M w^2\,d\mu
+ C\int_M\Bigl(\vert A(u)-A(v)\vert^2+\vert
b(u)-b(v)\vert^2\Bigr)\,d\mu\,.
\end{align*}
Estimating the last integral as before and putting the two computation together we obtain
\begin{align*}
\frac{d\,}{dt}\int_M\Big(\vert \nabla^pw\vert^2+ w^2\Big)\,d\mu \leq
&\,-\frac{\alpha}{2}\int_M \vert\nabla^{2p}w\vert^2\,d\mu
+C\sum_{r=0}^{2p-1}\int_M\vert\nabla^rw\vert^2\,d\mu\,.
\end{align*}
In order to deal with the last term, 
we apply the following Gagliardo--Nirenberg interpolation
inequalities (see~\cite[Proposition~2.11]{aubin0}
and~\cite[Theorem~4.14]{adams}): for every $0\leq r <2p$ and $\varepsilon>0$ there
exists a constant $C_\varepsilon$ such that 
$$
\Vert \nabla^r f\Vert_{L^2(M)}^2\leq \varepsilon\Vert
\nabla^{2p}f\Vert_{L^2(M)}^2
+C_\varepsilon\Vert f\Vert^2_{L^2(M)}
$$
for every function $f\in W^{2p,2}(M)$.\\
Hence, for some $\ve>0$ small enough we get,
\begin{align*}
\frac{d\,}{dt}\int_M\Bigl(\vert \nabla^pw\vert^2+
w^2\Bigr)\,d\mu
\leq&\,-\frac{\alpha}{4}\int_M \vert\nabla^{2p}w\vert^2\,d\mu
+C\sum_{r=0}^{2p-1}\varepsilon\int_M\vert\nabla^{2p}w\vert^2\,d\mu+C\sum_{r=0}^{2p-1}C_{\varepsilon}\int_Mw^2\,d\mu\\
\leq&\,C\int_M w^2\,d\mu\,.
\end{align*}
From this ordinary differential inequality and Gronwall's lemma, it
follows that if the quantity $\int_M(\vert \nabla^pw\vert^2+
w^2)\,d\mu$ is zero at some time $t_0$, then it must be zero
for every time $t\in [t_0,T]$. Since at $t=0$ we have 
$w(\,\cdot\,,0)=u_0-v_0=0$, we are done.

\medskip

We now prove the continuous dependence
of a solution $u\in C^\infty(M\times[0,T])$ on its initial
datum $u_0=u(\,\cdot\,,0)\in C^\infty(M)$. Fix any $m\in\NN$
satisfying condition~\eqref{2pm}, so that by the Sobolev embeddings 
$u\in P^m(M,T)$ implies $\nabla^{2p-1}u\in C^0(M\times[0,T])$. 
By the above argument, $u=(\M{F}|_U)^{-1}(u_0,0)\in P^m(M,T)$ 
where $\M{F}|_U$ is a diffeomorphism of an open set $U\subset
P^{m}(M,T)$ onto $V\subset W^{p(2m-1),2}(M)\times P^{m-1}(M,T)$, with
$(u_0,0)\in V$. Then, assuming that $u_{k,0}\to u_0$ in $C^\infty(M)$
as $k\to\infty$, we also have $u_{k,0}\to u_0$ in $W^{p(2m-1),2}(M)$,
hence for $k$ large enough $(u_{k,0},0)\in V$ and there exists $u_k\in
U$ 
such that $\M{F}(u_k)=(u_{k,0},0)$. This is the unique solution in $P^m(M,T)$ (hence in $C^\infty(M\times [0,T])$ by parabolic
bootstrap) with initial 
datum $u_{k,0}$. Moreover, since $\M{F}|_U$ is a diffeomorphism, we have $u_k\to u$ in $P^m(M,T)$.\\
By uniqueness, we can repeat the same
procedure for any $m\in\mathbb{N}$ satisfying condition~\eqref{2pm} 
concluding that $u_k\to u$ in $P^m(M,T)$ for every such
$m\in\mathbb{N}$, hence in $C^\infty(M\times [0,T])$.

\section{Proof of Lemma~\ref{F}}\label{lemmasec}

We shall write $P^m=P^m(M,T)$, $L^q=L^q(M\times [0,T])$,
$C^0=C^0(M\times [0,T])$ etc..., so that for instance $C^0(P^{m};C^1)$
will denote the space of continuous maps from $P^{m}(M,T)$
to $C^1(M\times [0,T])$. The first component of $\M{F}$, i.e. the map
$u\mapsto u(\,\cdot\,,0)$ is linear and bounded from $P^{m}$ to
$W^{p(2m-1),2}(M)$, by Proposition~\ref{teopolden2}, therefore it is 
$C^1$. Obviously the map $u\mapsto \de_t u$ is linear and
bounded from $P^{m}$ to $P^{m-1}$, hence also $C^1$. Thus, it
remains to show that the two maps
$$
\M{F}_A (u):= A(u)\cdot \nabla^{2p}u\,,\quad \M{F}_b(u):=b(u)
$$
belong to $C^1(P^{m};P^{m-1})$.

\medskip

We first prove that $\M{F}_A ,\M{F}_b\in C^0(P^{m};P^{m-1})$.
By an induction argument, it is easy to see that for every $k\in\NN$
\begin{equation}\label{dA2}
\nabla^k\big(A(u)\cdot\nabla^{2p} u\big)
=\sum_{j=0}^k\sum_{\substack {i_1,i_2,\dots,i_{j+1}\ge 1 \\i_1+\dots
    +i_{j+1}\le k+2p+(2p-1)j}}\de^j A(u) \# \nabla^{i_1}u \#\dots
\#\nabla^{i_{j+1}}u\,,
\end{equation}
where $\de^j
A(u)$ denotes the $j$--th derivative of $A$ with respect to any of
its arguments and $D\#E$ denotes an arbitrary contraction with
the metric of two tensors $D$ and $E$.

Taking into account formula~\eqref{dA2} with $k\leq 2p(m-1)$, in order to prove
that the map $u\mapsto \nabla^{2p(m-1)}(A(u)\cdot \nabla^{2p}u)$ 
belongs to $C^0(P^{m};L^2)$ we have to show that any map of the form
\begin{equation}\label{inL2}
u\mapsto \de^j A(u)\#\nabla^{i_1}u\#\cdots\#\nabla^{i_{j+1}}u
\end{equation}
belongs to $C^0(P^{m};L^2)$ whenever
\begin{equation}\label{condi}
i_1+\dots +i_{j+1}\le 2pm+(2p-1)j\,\quad\text{ and }\qquad i_1,\dots,i_{j+1}\ge 1\,.
\end{equation}
The case $r=0$ and $\ell ={2p-1}$ of the Sobolev
embeddings~\eqref{sobo4} below and condition~\eqref{2pm} imply that
if $u\in P^m$ then $\nabla^{2p-1}u\in C^0$ (and the immersion is bounded), hence all the maps
$u\mapsto \de^j A(u)$ belong to $C^0(P^{m};C^0)$.\\
We can assume from now on that $j\geq 1$, since in the case $j=0$, 
we get the term $A(u)\#\nabla^{2p+k} u$ which 
is continuous from $P^m$ to $L^2$ as a function of $u$ for
$k\leq 2p(m-1)$.\\
As for the factors $\nabla^{i_\ell}u$ appearing in formula~\eqref{inL2}, first we assume that each
$i_\ell$ is such that we are in case~\eqref{sobo3} of Sobolev embeddings, i.e.
\begin{equation}\label{condsobo}
\frac{1}{q_\ell}:=\frac{1}{2}-\frac{2pm-i_\ell}{n+2p}>0,
\end{equation}
so that the map $u\mapsto \nabla^{i_\ell} u$ lies in $C^0(P^{m};L^{q_{\ell}})$.
By H\"older's inequality, the condition
\begin{equation}\label{contoalg0}
\frac{1}{q}:=\sum_{\ell=1}^{j+1}\frac{1}{q_{\ell}}=\sum_{\ell=1}^{j+1}\bigg(\frac{1}{2}-\frac{2pm-i_\ell}{n+2p}\bigg)\le\frac{1}{2}\,,
\end{equation}
implies that the map $u\mapsto\nabla^{i_1}u\#\cdots\#\nabla^{i_{j+1}}u$ belongs to $C^0(P^m;L^{q})$, hence also to $C^0(P^m;L^2)$, as $L^q$ embeds 
continuously into $L^2$ for $q\ge 2$. Then, if we show 
inequality~\eqref{contoalg0}, the map defined by formula~\eqref{inL2}
belongs to $C^0(P^{m};L^2)$.
From inequalities~\eqref{2pm}, \eqref{condi} and $j\ge 1$ it follows,
\begin{equation}\label{contoalg}
\sum_{\ell=1}^{j+1}\frac{1}{q_{\ell}}\le \frac{j+1}{2}-\frac{2pm(j+1)-2pm-(2p-1)j}{n+2p}
=\frac{1}{2}+\frac{j}{2}-\frac{(2pm-2p+1)j}{n+2p}
<\frac{1}{2}\,.
\end{equation}

Now, if for some $i_\ell$, say $i_1,\ldots,i_s$, 
we have $\frac{2pm-i_\ell}{n+2p}>\frac{1}{2}$, then we are in
case~\eqref{sobo4} of Sobolev embeddings and the corresponding maps $u\mapsto \nabla^{i_\ell}u$ belong to
$C^0(P^{m};C^0)$, hence we can avoid to estimate such factors, as
for $A(u)$. Then, since~\eqref{condsobo} holds for
$\ell\in\{s+1,\dots,j+1\}$, arguing again by
induction, in this case 
we have to deal with functions $u\mapsto \nabla^{i_{s+1}}u\#\cdots
\#\nabla^{i_{j+1}}u$ under the conditions
$$
i_{s+1}+\dots +i_{j+1}\le 2pm+(2p-1)(j-s)\,\quad\text{ and
}\qquad i_{s+1},\dots i_{j+1}\ge 1\,.
$$
Then, computing as in inequality~\eqref{contoalg} one shows
\begin{align}\label{contoalg2}
\sum_{\ell=s+1}^{j+1}\frac{1}{q_{\ell}}
&\,\le \frac{j+1-s}{2}-\frac{2pm(j+1-s)-2pm-(2p-1)(j-s)}{n+2p}\\
&\,=\frac{1}{2}+\frac{j-s}{2}-\frac{(2pm-2p+1)(j-s)}{n+2p}\nonumber\\
&\,\leq\frac{1}{2}\nonumber\,,
\end{align}
where we intend that if $s=j+1$ there is nothing to sum. Notice
that the last inequality is strict if $s\not=j$, and in the case 
$s=j$ the map $u\mapsto\nabla^{i_{j+1}}u$ is continuous
from $P^m$ to $L^2$ as $i_{j+1}\leq 2pm$.

If in addition for some $i_\ell$, say $i_{s+1},\ldots, i_r$, 
we have $\frac{2pm-i_\ell}{n+2p}=\frac{1}{2}$ (i.e. 
we are in the critical case~\eqref{sobo3b} of the Sobolev
embeddings), we know that for such indices the maps 
$u \mapsto \nabla^{i_\ell}u$ belong to $C^0(P^{m};L^q)$ for every
$1\le q<\infty$. Then inequality~\eqref{contoalg2} 
still holds true if we choose $q_{s+1},\ldots, q_r$ large enough,
since, unless $s=r=j$, the last inequality in \eqref{contoalg2} is strict.\\
Hence, we conclude as before that the map $u\mapsto \nabla^{2p(m-1)}(A(u)\cdot
\nabla^{2p}u)$ lies in $C^0(P^{m};L^2)$. 

The time or mixed space-time
derivatives $\partial_t^r\nabla^k(A(u)\cdot
\nabla^{2p}u)$ with $2pr+k\leq 2p(m-1)$ can be treated in a similar
way, observing that the functions 
$\de_t^r \nabla^\ell u$ have the same integrability of
$\nabla^{2pr+\ell}u$ from the point of view of the
embeddings~\eqref{sobo3}--\eqref{sobo4}.\\
Starting from formula~\eqref{dA2} and differentiating in time, again
by an induction argument, one gets
\begin{equation}\label{gener}
\,\de_t^r \nabla^k\big(A(u)\cdot\nabla^{2p} u\big) 
\,=\sum_{j=0}^{r+k}\sum_{\substack {i_1,\dots,i_{j+1},
    \iota_1,\dots,\iota_{j+1}\ge 0\\ i_1+\dots +i_{j+1}\le
    k+2p+(2p-1)j\\ \iota_1+\dots+\iota_{j+1}\le r }}\de^j
A(u)\#\de_t^{\iota_1}\nabla^{i_1}u\#\cdots
\#\de_t^{\iota_{j+1}}\nabla^{i_{j+1}}u\,.
\end{equation}
Then, with the same proof as before one shows that a map of the form
$$
u\mapsto \de^j A(u)\#\de_t^{\iota_1}\nabla^{i_1}u\#\cdots
\#\de_t^{\iota_{j+1}}\nabla^{i_{j+1}}u
$$
belongs to $C^0(P^{m+1};L^2)$ whenever $i_1,\dots, i_{j+1},\iota_1,\dots,\iota_{j+1}\ge 0$ and
\begin{equation}\label{condib}
i_1+\dots +i_{j+1}+2p(\iota_1+\dots+\iota_{j+1})\le
2pm+(2p-1)j\,.
\end{equation}
Hence the map $u\mapsto \de_t^r\nabla^k\big(A(u)\cdot \nabla^{2p}u\big)$
belongs to $C^0(P^{m};L^2)$ for $2pr+k \le 2p(m-1)$, 
which means that $\M{F}_A\in C^0(P^{m};P^{m-1})$ as wished.

The map $\M{F}_b$ can be treated in a similar way, so also $\M{F}_b \in C^0(P^{m};P^{m-1})$.

\medskip

It remains to prove that $d\M{F}_A, d\M{F}_b\in C^0(P^m;
L(P^m;P^{m-1}))$, where $L(P^m;P^{m-1})$ denotes the Banach space of
bounded linear maps from $P^m$ into $P^{m-1}$. We first claim that the
Gateaux derivative
\begin{equation}\label{diffk}
(u,v )\mapsto d\M{F}_A(u)(v):= \frac{d}{dt}\M{F}_A(u+tv)\Big|_{t=0}
\end{equation}
belongs to $C^0(P^{m}\times P^m;P^{m-1})$. Indeed, $d\M{F}_A(u)(v)$ can be written as
$$
B(u,v)\#\nabla^{2p}u + A(u)\cdot \nabla^{2p}v\,,
$$
where $B$ is a tensor depending smoothly on
$x,t,u,\ldots,\nabla^{2p-1}u$ and linearly on some derivative of $v$
up to the order $2p-1$, that is,
$B(u,v)=\sum_{\ell=0}^{2p-1}B_\ell(u)\cdot \nabla^\ell v$, 
compare with formula~\eqref{dfeq}.
The estimates proven for $\M{F}_A$ can be applied to any term of the
form $\de_t^r\nabla^k(B(u,v)\#\nabla^{2p}u)$, since they can be expressed as a
sum similar to the right--hand side of identity~\eqref{gener}. The
only difference is that now in every term of such sum one linear 
occurrence of $u$ is replaced by $v$. Precisely, writing $u_1:= u$, $u_2:=v$ every term 
$\partial^j A(u)\#\de_t^{\iota_1}\nabla^{i_1}u\#\cdots
\#\de_t^{\iota_{j+1}}\nabla^{i_{j+1}}u$ has to be replaced by some
\begin{equation}\label{tau}
D(u)\#\de_t^{\iota_1}\nabla^{i_1}u_{\tau_1}\#
\cdots\#\de_t^{\iota_{j+1}}\nabla^{i_{j+1}}u_{\tau_{j+1}}
\end{equation}
where exactly one of the indices $\tau_1,\dots,\tau_{j+1}$ is equal to $2$, and the others are equal to $1$.\\
An analogous reasoning applies to the term $A(u)\cdot \nabla^{2p}v$. 
It is then easy to see, since $v\in P^m$ like $u$, that we can repeat the
same estimates used to show the continuity of
$u\mapsto\M{F}_A(u)$. This proves in particular that 
$d\M{F}_A(u)\in L (P^m;P^{m-1})$.\\
In order now to prove that $d\M{F}_A\in C^0(P^m;L(P^m;P^{m-1}))$ we need
to show that
$$
\sup_{\|v\|_{P^m}\le 1}\|d\M{F}_A({\widetilde
  u})(v)-d\M{F}_A(u)(v)\|_{P^{m-1}}\to 0\quad \text{as }{\widetilde
  u}\to u\text{ in }P^m\,.
$$
Again, this estimate is similar to what we have already done. Indeed, supposing that $\tau_{j+1}$ is the only index equal to
2 in~\eqref{tau} and assuming that there are no time derivatives
for the sake of simplicity, we want to see that, 
as ${\widetilde  u}\to u$ in $P^m$,
\begin{equation}\label{Btilde}
\sup_{\|v\|_{P^m}\le 1}   \|D({\widetilde u})\#\nabla^{i_1}{\widetilde u}\#
\cdots\#\nabla^{i_{j}}{\widetilde u}\,\nabla^{i_{j+1}}v
-D(u)\#\nabla^{i_1}u\#
\cdots\#\nabla^{i_{j}}u\,\nabla^{i_{j+1}}v\|_{L^2}\to 0\,,
\end{equation}
where $i_1+\dots+i_{j+1}\leq 2pm+(2p-1)j$ (see formula~\eqref{dA2} and 
condition~\eqref{condi}).\\
Adding and subtracting terms, one gets
\begin{align*}
\Bigl\vert\,D({\widetilde u})\#\nabla^{i_1}{\widetilde u}\#
\cdots\#\nabla^{i_{j}}{\widetilde u}\,\nabla^{i_{j+1}}v
-&\,D(u)\#\nabla^{i_1}u\#
\cdots\#\nabla^{i_{j}}u\,\nabla^{i_{j+1}}v\,\Bigr\vert\\
\le\Bigl\{&\,\vert D({\widetilde u})-D(u)\vert\,\vert\nabla^{i_1}\widetilde{u}\vert\,\cdots\,\vert\nabla^{i_{j}}\widetilde{u}\vert\\
&\,+\vert D(u)\vert\,\vert\nabla^{i_1}({\widetilde
  u}-u)\vert\,\vert\nabla^{i_2}\widetilde{u}\vert\,\cdots\,\vert\nabla^{i_{j}}\widetilde{u}\vert\\
&\,+\,\cdots\,+\vert D(u)\vert\,\vert\nabla^{i_1}u\vert\,\cdots\,\vert\nabla^{i_{j}}({\widetilde
  u}-u)\vert\Bigr\}\,\vert\nabla^{i_{j+1}}v\vert\,.
\end{align*}
Studying now the $L^2$ norm of this sum, the first term can be bounded
as before and it goes to zero as $D(u)$ is continuous from $P^m$ to
$L^\infty$. The $L^2$ norm of all the other terms, repeating step by step the previous
estimates, using H\"older's inequality and
embeddings~\eqref{sobo3}--\eqref{sobo4}, 
will be estimated by some product
$$
C\Vert u\Vert^\alpha_{P^m}\Vert \widetilde{u}\Vert^\beta_{P^m}\Vert
v\Vert^\gamma_{P^m}\Vert\widetilde{u}-u\Vert^\sigma_{P^m}\leq
C\Vert u\Vert^\alpha_{P^m}\Vert
\widetilde{u}\Vert^\beta_{P^m}\Vert\widetilde{u}-u\Vert^\sigma_{P^m}
$$
for a constant $C$ and some nonnegative exponents $\alpha,\beta,\gamma,\sigma$
satisfying $\alpha+\beta+\gamma+\sigma\leq 1$ and $\sigma>0$. Here we we used
the fact that $\Vert v\Vert_{P^m}\leq 1$.\\
As ${\widetilde u}-u\to 0$ in $P^m$, this last product goes to zero in
$L^2$, hence uniformly for $\Vert v\Vert_{P^m}\le 1$ and
inequality~\eqref{Btilde} follows, as claimed.
The analysis of the estimates with mixed time/space derivatives is
similar and all this argument works analogously for the term $A(u)\cdot \nabla^{2p}v$.\\
Then, the Gateaux derivative $d\M{F}_A$ is continuous which implies that it coincides with the Frech\'et derivative, hence $\M{F}_A\in C^1(P^m;P^{m-1})$.

The map $\M{F}_b$ can be dealt with in the same way and we are done.

\section{Parabolic Sobolev Embeddings}

\begin{prop}\label{propsobo} Let $u\in P^m(M,T)$. Then for
  $r,\ell\in\mathbb{N}$ with $2pr+\ell\le 2mp$, we have
\begin{align}\label{sobo3}
\|\de_t^r\nabla^\ell u\|_{L^q(M\times[0,T])}\le C
\|u\|_{P^m(M,T)}\quad
&\text{if}\quad\frac{1}{q}=\frac{1}{2}-\frac{2pm-\ell-2pr}{n+2p}>0\,;\\
\label{sobo3b}
\|\de_t^r\nabla^\ell u\|_{L^q(M\times[0,T])}\le C
\|u\|_{P^m(M,T)}\quad
&\text{if}\quad\frac{1}{2}-\frac{2pm-\ell-2pr}{n+2p}=0\,\text{ and }\,
1\le q<\infty\,;
\end{align}
the function $\de_t^r\nabla^\ell u$ is continuous and 
\begin{align}\label{sobo4}
\|\de_t^r\nabla^\ell u\|_{C^0(M\times[0,T])}\le C
\|u\|_{P^m(M,T)} \quad &\text{if} \quad
\frac{1}{2}-\frac{2pm-\ell-2pr}{n+2p}<0\,,\phantom{\,\text{ and }\,
1\le q<\infty}
\end{align}
where the constant $C$ does not depend on $u$.
\end{prop}

\begin{proof} Of course we can write
$$
P^m(M,T)= L^2([0,T]; H^{2mp}(M))\cap H^1([0,T];
H^{2p(m-1)}(M))\cap \dots\cap H^m([0,T]; L^2(M))\,.
$$
By standard interpolation theory, see e.g.~\cite[Theorem~2.3]{lm},
we have the continuous immersion
$$
P^m(M,T) \hookrightarrow H^s([0,T];H^{2p(m-s)}(M)),\quad
\text{for all }s\in [0,m]\,.
$$
We shall now assume that $\frac{1}{2}-\frac{2pm-\ell-2pr}{n+2p}>0$ and
prove inequality~\eqref{sobo3}. For $0\le \sigma<\frac{1}{2}$ and for any Hilbert
space $X$ we have the Sobolev embedding
$$
H^\sigma([0,T];X)\hookrightarrow L^q([0,T];X)\quad \text{for}\quad
\frac{1}{q}=\frac{1}{2}-\sigma\,.
$$
Then, for $\ell,r\in\mathbb{N}$ with $2pr+\ell\le 2pm$ and for any
$s\in
\big(m-\frac{\ell}{2p}-\frac{n}{4p},m-\frac{\ell}{2p}\big]\cap\big[r,r+\frac{1}{2}\big)$,
also using the standard Sobolev embeddings on $M$, for every $u\in P^m(M,T)$ one gets
\begin{equation*}
\begin{split}
\de_t^r \nabla^\ell u\in
H^{s-r}([0,T];H^{2p(m-s)-\ell}(M))&\hookrightarrow
L^q([0,T];H^{2p(m-s)-\ell}(M)) \\
&\hookrightarrow L^q([0,T];L^{\widetilde{q}}(M))\,,
\end{split}
\end{equation*}
with
$$
\frac{1}{q}=\frac{1}{2}-s+r\quad\text{ and }\quad\frac{1}{\widetilde{q}}=\frac{1}{2}-\frac{2p(m-s)-\ell}{n}\,.
$$
We now choose $s=\frac{rn+2pm-\ell}{n+2p}$ and claim that $s\in
\big(m-\frac{\ell}{2p}-\frac{n}{4p},m-\frac{\ell}{2p}\big]\cap\big[r,r+\frac{1}{2}\big)$.
Then
$$
\frac{1}{q}=\frac{1}{\widetilde{q}}=\frac{1}{2}-\frac{2pm-\ell-2pr}{n+2p}\,,
$$
hence for such $q\in\RR$ we have
$$
u\in L^q([0,T];L^q(M))\simeq L^q(M\times[0,T])\,,
$$
and embedding~\eqref{sobo3} is proven. As for the claim, the
inequalities $s\ge r$ and $s\le m-\frac{\ell}{2p}$ easily follow from
the inequality $2pr+\ell\le 2pm$, while inequality $s<r+\frac{1}{2}$ is
equivalent to $\frac{1}{2}-\frac{2pm-\ell-2pr}{n+2p}>0$. This means 
$\frac{1}{q}>0$ which implies $s> m-\frac{\ell}{2p}-\frac{n}{4p}$.

The proof of inequality~\eqref{sobo3b} is analogous.

Finally, if $\frac{1}{2}-\frac{2pm-\ell-2pr}{n+2p}<0$, using that
for $\sigma>\frac{1}{2}$ one has $H^\sigma([0,T];X)\hookrightarrow
C^0([0,T];X)$ and that for $\sigma>\frac{n}{2}$ one has
$H^\sigma(M)\hookrightarrow C^0(M)$, for every $u\in P^m(M,T)$ we infer
$$
\de_t^r\nabla^\ell u\in H^{s-r}([0,T];H^{2p(m-s)-\ell}(M))\hookrightarrow
C^0([0,T];C^0(M))\simeq C^0(M\times[0,T])\,,
$$
for $s=\frac{rn+2pm-\ell}{n+2p}\in \big(r+\frac{1}{2},
m-\frac{\ell}{2p}-\frac{n}{4p}\big)$. This proves embedding~\eqref{sobo4}.
\end{proof}

\bibliographystyle{amsplain}
\bibliography{biblio}

\end{document}